\documentclass[3p,10pt,a4paper,twoside,fleqn,sort&compress]{article}

\usepackage{amsmath,amsthm,amsfonts,amssymb,latexsym,mathrsfs,color,pxfonts}

\newtheorem{thm}{Theorem}[section]
\newtheorem{lem}{Lemma}[section]
\newtheorem{prop}{Proposition}[section]
\newtheorem{conj}{Conjecture}[section]
\newtheorem{prob}{Problem}[section]

\theoremstyle{definition}

\newtheorem{f}{Fact}[section]
\newtheorem{rem}{Remark}[section]

\numberwithin{equation}{section} \numberwithin{equation}{section}

\newcommand{\la}{\lambda}

\title{Unimodality of the independence
polynomials of some composite graphs
\thanks{Supported partially by the National Natural Science Foundation of China (Nos.
11201191, 11171150, 11171288).\newline\hspace*{5mm}
   {\it Email address:}\quad bxzhu@jsnu.edu.cn (B.-X. Zhu), qllu@jsnu.edu.cn (Q.-L. Lu)} }
\author{Bao-Xuan Zhu$^1$\footnote{Corresponding author.},\quad Qinglin
   Lu$^1$}
\date{\footnotesize 1. School of Mathematics and Statistics,
         Jiangsu Normal University,
         Xuzhou 221116, PR China}
\begin{document}
\maketitle
\begin{abstract}
Let  $I(G;x)$  denote the independence polynomial of a graph $G$. In
this paper we study the unimodality properties of $I(G;x)$ for some
composite graphs $G$.

Given two graphs $G_1$ and $G_2$, let $G_1[G_2]$ denote the
lexicographic product of $G_1$ and $G_2$. Assume
$I(G_1;x)=\sum_{i\geq0}a_ix^i$ and $I(G_2;x)=\sum_{i\geq0}b_ix^i$,
where $I(G_2;x)$ is log-concave. Then we prove (i) if $I(G_1;x)$ is
log-concave and $(a^2_i-a_{i-1}a_{i+1})b^2_1\geq a_ia_{i-1}b_2$ for
all $1\leq i \leq \alpha(G_1)$, then $I(G_1[G_2];x)$ is log-concave;
(ii) if $a_{i-1}\leq b_1a_i$ for $1\leq i\leq \alpha(G_1)$, then
$I(G_1[G_2];x)$ is unimodal. In particular, if $a_i$ is increasing
in $i$, then $I(G_1[G_2];x)$ is unimodal. We also give two
sufficient conditions when the independence polynomial of a complete
multipartite graph is unimodal or log-concave. Finally, for every
odd positive integer $\alpha > 3$, we find a connected graph $G$ not
a tree, such that $\alpha(G) =\alpha$, and $I(G; x)$ is symmetric
and has only real zeros. This answers a problem of Mandrescu and
Miric\u{a}.
\end{abstract}

\section{Introduction}
A graph polynomial is an algebraic object associated with a graph
that is usually invariant at least under graph isomorphism. As such,
it encodes information about the graph, and enables algebraic
methods for extracting this information. Graph polynomials are
widely studied, e.g., Tutte polynomial, chromatic polynomial,
matching polynomial, independence polynomial, and so on, which have
been found many applications in chemistry and physics.

 Let $G=(V(G),E(G))$
be a finite and simple graph. An {\it independent set} in a graph
$G$ is a set of pairwise non-adjacent vertices. A {\it maximum
independent set} in $G$ is a largest independent set and its size is
denoted by $\alpha(G)$. Let $i_k(G)$ denote the number of
independent sets of cardinality $k$ in $G$. Then its generating
function
$$I(G;x)=\sum\limits_{k=0}^{\alpha(G)}i_k(G)x^k,\quad i_0(G)=1$$
is called the {\it independence polynomial} of $G$ (Gutman and
Harary~\cite{GH83}). It is clear that $i_1(G)=|V(G)|$ and
$i_2(G)=\binom{|V(G)|}{2}-|E(G)|.$  For $v\in V(G)$, let $N(v)=\{w:
vw\in E(G)\}$ and $N[v]=N(v)\cup \{v\}$. The following is
fundamental:
 $$I(G;x)=I(G-v;x)+xI(G-N[v];x)$$ for arbitrary $v\in V(G)$, see \cite{GH83}.

A polynomial $\sum_{k=0}^na_kx^k$ with nonnegative coefficients is
called {\it unimodal} if there is some $m$, such that
$$a_0\le a_1\le\cdots\le a_{m-1}\le a_m\ge a_{m+1}\ge\cdots\ge a_n;$$
it is called {\it symmetric} if $a_k=a_{n-k}$ for $0\le k\le \lfloor
n/2\rfloor$; it is called {\it log-concave} if $a_k^2\ge
a_{k-1}a_{k+1}$ for all $1\le k \le n-1$; it is {\it strictly
log-concave} if $a_k^2> a_{k-1}a_{k+1}$ for all $1\le k \le n-1$. It
is known that a log-concave polynomial with positive coefficients is
unimodal. A basic approach to unimodality problems is to use
Newton's inequalities: Let $a_0,a_1,\ldots,a_n$ be a sequence of
nonnegative numbers. Suppose that the polynomial
$\sum_{k=0}^{n}a_kx^k$ has only real zeros. Then
$$a_k^2\ge a_{k-1}a_{k+1}\left(1+\frac{1}{k}\right)\left(1+\frac{1}{n-k}\right),\quad k=1,2,\ldots,n-1,$$
and the sequence is therefore log-concave and unimodal (see Hardy,
Littlewood and P\'olya~\cite[p. 104]{HLP52}). Unimodality problems
arise naturally in many branches of mathematics and have been
extensively investigated. See Stanley's survey~\cite{Sta89} and
Brenti's supplement~\cite{Bre94} for known results and open problems
on log-concavity and unimodality arising in algebra, combinatorics
and geometry.

Unimodality problems of independence polynomials have attracted
researchers' great interest, see
\cite{AMSE87,BHN04,BN05,CS07,LM05,LM06EJC,WZ10,Zhu13} for instance.
Alavi, Malde, Schwenk, Erd\H{o}s~\cite{AMSE87} found that
independence polynomials are not unimodal in general and conjectured
the following.

\begin{conj}\label{conj+tree}
The independence polynomial of any tree or forest is unimodal.
\end{conj}
This conjecture is still open. In general, the independence
polynomial of a graph may be neither log-concave nor unimodal, as
evidenced by the graph $G=3K_{4}+K_{37}$ with
$I(G;x)=1+49x+48x^2+64x^3$. But the independence polynomials for
certain special classes of graphs are unimodal and even have only
real zeros. For instance, the independence polynomial of a line
graph has only real zeros \cite{HL72}. More generally, the
independence polynomial of a claw-free graph has only real
zeros~\cite{CS07}. Thus, a natural problem arises.

\begin{prob}\label{problem} Which special class of graphs have unimodal independence
polynomials ?
\end{prob}

Recently, by researching the operations on graphs, there has been
some partial results for Problem~\ref{problem}, see Bahls
\cite{B10}, Bahls and Salazar \cite{BS10}, Levit and Mandrescu
\cite{LM04WSEAS}, Mandrescu \cite{Man09}, Wang and Zhu \cite{WZ10}
and Zhu \cite{Zhu13} for instance. Motivated by Problem
\ref{problem}, we will give some products of graphs having unimodal
independence polynomials, including the rooted product of graphs and
lexicographic product of graphs. On the other hand, note that the
complete multipartite graphs are important and familiar. However,
there are fewer known results for the unimodality of their
independence polynomials. Therefore, we also study the unimodality
of independence polynomials of the complete multipartite graphs.

Recently, Mandrescu and Miric\u{a}~\cite{MM11} found for every
integer $2\leq \alpha\neq 3$ there is a forest $F$ consisting of at
most two non-trivial trees, whose $\alpha(F)=\alpha$, and $I(F; x)$
is symmetric and has only real zeros. They further proposed the
following problem.

\begin{prob}\label{pro} For every odd positive integer $\alpha > 3$, find a connected graph
$G$ different from a tree, such that $\alpha(G) =\alpha$, and $I(G;
x)$ is symmetric and has only real zeros.
\end{prob}
In this paper, we also answer this problem by finding a connected
bipartite graph.


\section{Lexicographic product of graphs}
To simplify our proof, we need the next result, which is very useful
in solving unimodality problems for polynomials.
\begin{lem}\emph{\cite{Sta89}}\label{product}
Let $f(x)$ and $g(x)$ be polynomials with positive coefficients.
\begin{itemize}
\item [\rm (i)] If both $f(x)$ and $g(x)$ are log-concave, then
so is their product $f(x)g(x)$.
\item [\rm (ii)] If $f(x)$ is log-concave and $g(x)$ is unimodal, then
their product $f(x)g(x)$ is unimodal.
\item [\rm (iii)] If both $f(x)$ and $g(x)$ have only real zeros, then
so does their product $f(x)g(x)$.
\end{itemize}
\end{lem}
Recall the definition of lexicographic product of graphs. For two
graphs $G_1$ and $G_2$, let $G_1[G_2]$ be the graph with vertex set
$V(G_1)\times V(G_2)$ and such that a vertex $(a, x)$ is adjacent to
a vertex $(b, y)$ if and only if $a$ is adjacent to $b$ (in $G_1$)
or $a=b$ and $x$ is adjacent to $y$ (in $G_2$). The graph $G_1[G_2]$
is called the {\it lexicographic product} (or composition) of $G_1$
and $G_2$, and can be thought of as the graph arising from $G_1$ and
$G_2$ by substituting a copy of $G_2$ for every vertex of $G_1$. In
\cite{BHN04}, it was proved that
\begin{equation}\label{eq+lex}
I(G_1[G_2];x)=I(G_1;I(G_2;x)-1). \end{equation} Motivated by
(\ref{eq+lex}), we prove the following general result, which can be
well applied to the independence polynomial of the lexicographic
product of graphs. We refer readers to \cite{CYZ10,LMB10,WYeujc05}
for some similar results.
\begin{thm}\label{thm+com}
Let polynomials $f(x)=\sum_{i=0}^{n}a_ix^i$ and
$g(x)=\sum_{i=1}^{m}b_ix^i$ with positive coefficients be given.
\begin{itemize}
\item [\rm (i)]
Assume that both $f(x)$ and $g(x)$ are log-concave. If
$(a^2_i-a_{i-1}a_{i+1})b^2_1\geq a_ia_{i-1}b_2$ for all $1\leq i
\leq n$, then $f(g(x))$ is log-concave;
\item [\rm (ii)]
Assume that $g(x)$ is log-concave. If $a_{i-1}\leq b_1a_i$ for
$1\leq i\leq n$, then $f(g(x))$ is unimodal. In particular, if the
sequence $a_n$ is increasing in $n$ and $b_1\geq1$, then $f(g(x))$
is unimodal.
\end{itemize}
\end{thm}
\begin{proof}
Let $f(g(x))=\sum_{i=0}^{mn}c_ix^i$.

(i) Note that it is trivial for $n=0$. In the following, we will
prove (i) by induction on $n$. If $n=1$, then
$$f(g(x))=a_0+a_1b_1x+a_1b_2x^2+\ldots+a_1b_mx^m.$$
By the hypothesis, its log-concavity follows from $a^2_{1}b^2_1\geq
a_{1}a_0b_2$. So we proceed to the inductive step.

Let $F(x)=\sum_{i=0}^{n-1}a_{i+1}x^i$. Then
$f(g(x))=a_0+g(x)F(g(x))$. By the induction hypothesis, $F(g(x))$ is
log-concave. So $g(x)F(g(x))$ is log-concave by Lemma~\ref{product}
(i). Thus $c_1,c_2,\ldots,c_{mn}$ is log-concave. To show the
log-concavity of $f(g(x))$, it suffices to check $c_1^2\geq c_0c_2$,
which follows from the hypothesis since $c_0=a_0$, $c_1=a_1b_1$ and
$c_2=a_1b_2+a_2b_1^2$.

(ii) Similarly, we will prove (ii) by induction on $n$. If $n=1$,
then
$$f(g(x))=a_0+a_1b_1x+a_1b_2x^2+\ldots+a_1b_mx^m.$$ Since $g(x)$ is
log-concave, we have $b_1,b_2,\ldots,b_m$ is unimodal. Thus, it
follows from $a_{1}b_1\geq a_0$ that $f(g(x))$ is unimodal. So we
proceed to the inductive step.

Let $F(x)=\sum_{i=0}^{n-1}a_{i+1}x^i$. Then
$f(g(x))=a_0+g(x)F(g(x))$. By the induction hypothesis, $F(g(x))$ is
unimodal. So $g(x)F(g(x))$ is unimodal by Lemma~\ref{product} (ii).
Thus $c_1,c_2,\ldots,c_{mn}$ is unimodal. To show the unimodality of
$f(g(x))$, it suffices to check $c_1\geq c_0$, which follows from
the hypothesis since $c_0=a_0$ and $c_1=a_1b_1$.

This completes the proof.
\end{proof}

 By Theorem~\ref{thm+com} and (\ref{eq+lex}),
we have the next result for the independence polynomial of graphs.
\begin{thm}\label{prop}
For two vertex disjoint graphs $G_1$ and $G_2$, let
$I(G_1;x)=\sum_{i=0}^{\alpha(G_1)}a_ix^i$ and
$I(G_2;x)=\sum_{i=0}^{\alpha(G_2)}b_ix^i$.
\begin{itemize}
\item [\rm (i)]
Assume that $I(G_1;x)$ and $I(G_2;x)$ are log-concave. If
$(a^2_i-a_{i-1}a_{i+1})b^2_1\geq a_ia_{i-1}b_2$ for all $1\leq i
\leq \alpha(G_1)$, then $I(G_1[G_2];x)$ is log-concave;
\item [\rm (ii)]
Assume that $I(G_2;x)$ is log-concave. If $a_{i-1}\leq b_1a_i$ for
$1\leq i\leq \alpha(G_1)$, then $I(G_1[G_2];x)$ is unimodal. In
particular, if $a_i$ is increasing in $i$, then $I(G_1[G_2];x)$ is
unimodal.
\end{itemize}
\end{thm}

\begin{rem}
Let $|V(G_2)|=p$ and $|E(G_2)|=q$. Then we know that $b_1=p$ and
$b_2=\binom{p}{2}-q$. If $\frac{p^2}{\binom{p}{2}-q}$ is enough
large and $I(G_1;x)$ is strictly log-concave, then we can obtain
$(a^2_i-a_{i-1}a_{i+1})b^2_1\geq a_ia_{i-1}b_2$ for all $1\leq i
\leq \alpha(G_1)$. Thus, $I(G_1[G_2];x)$ is log-concave when
$I(G_2;x)$ is log-concave. On the other hand, if $p$ is
sufficiently  large, then we can obtain $a_{i-1}\leq b_1a_i$ for
$1\leq i\leq \alpha(G_1)$. Thus, $I(G_1[G_2];x)$ is unimodal when
$I(G_2;x)$ is log-concave.
\end{rem}
\begin{rem}\label{rem}
Let $G_1=G[K_p]$. If $p$ is sufficiently  large, then $I(G_1;x)$ is
nondecreasing. Thus, $I(G_1[G_2];x)$ is unimodal when $I(G_2;x)$ is
log-concave and $|V(G_2)|$ is sufficiently large.
\end{rem}
\begin{rem}
In the above results, the condition of the log-concavity can be
easily obtained if its independence polynomial has only real zeros
(for instance, for any claw-free graph).
\end{rem}

A graph is called well-covered if all its maximal independent sets
are of the same cardinality \cite{Pl70}. If graphs $G_1$ and $G_2$
are well covered, then so is $G_1[G_2]$, see \cite{BHN04}. Note that
it was proved for a well-covered graph that
$$i_{k-1}(G)\leq ki_k(G)$$
for $1\leq k\leq \alpha(G)$~\cite{BDN00}. Thus, by
Theorem~\ref{prop} (ii), we deduce the following.

\begin{prop}\label{prop+well}
Let $G_1$ and $G_2$ be two well-covered graphs. If $I(G_2;x)$ is
log-concave and $|V(G_2)|\geq \alpha(G_1)$, then $I(G_1[G_2];x)$ is
unimodal. In particular, if $I(G_2;x)$ is log-concave, then
$I(G_2[G_2];x)$ is unimodal.
\end{prop}
\begin{rem}
Noting that for any graph $G$, the rooted product $G\overline{\circ}
P_2$ of $G$ and $P_2$ (denote the path with two vertices) is a well
covered graph with $\alpha(G\overline{\circ} P_2)=|V(G)|$. So if $G$
is claw-free, $I(G\overline{\circ} P_2;x)$ has only real zeros since
 $I(G;x)$ has only real zeros, see Levit and Mandrescu
\cite{LM07}. Thus, let $G'=G\overline{\circ} P_2$, and by the above
Proposition \ref{prop+well} we get that $I(G'[G'];x)$ is unimodal.
Similarly, we can obtain more results.
\end{rem}

For the unimodality of independence polynomials of well-covered
graphs, we refer readers to \cite{BDN00,LM06EJC,LM071,LM07,MT03} for
details.

\section{Complete Multipartite Graphs}
Denote the complete $k$-partite graph by $K_{n_1,n_2,\ldots,n_k}$.
Then its independence polynomial is
\begin{equation*}
I(K_{n_1,n_2,\ldots,n_k};x)=\sum_{i=1}^{k}(1+x)^{n_i}-(k-1).
\end{equation*}
So if $K_{n_1,n_2,\ldots,n_k}$ has $a_i$ classes of size $i$ for
each $1\leq i\leq n$, then
\begin{equation}\label{comp}
I(K_{n_1,n_2,\ldots,n_k};x)=\sum_{i=1}^{n}a_i(1+x)^{i}-(k-1).\end{equation}
Note that unimodality or log-concavity of
$\sum_{i=1}^{n}a_i(1+x)^{i}$ implies that of
$I(K_{n_1,n_2,\ldots,n_k};x)$. If $k=2$ and $n_1\geq n_2$, then it
is easy to obtain that $(1+x)^{n_1}[(1+x)^{n_2-n_1}+1]$ is
log-concave by Lemma~\ref{product} (i). It follows that
$I(K_{n_1,n_2};x)$ is log-concave. In general, we have the following
result.
\begin{thm}
Assume that $G$ is a complete $k$-partite graph of order $n$ and
$k\geq 3$ and its independence polynomial satisfies (\ref{comp}).
\begin{itemize}
\item [\rm (i)]
If the sequence $\{a_i\}$ is positive and log-concave, then $I(G;x)$
is log-concave;
\item [\rm (ii)]
If the subsequence $\{a_i:a_i\neq0\}$ is increasing, then $I(G;x)$
is unimodal.
\end{itemize}
\end{thm}
\begin{proof}
(i) directly follows from the result that if a positive sequence
$\{d_i\}_{i=0}^n$ is log-concave then so is the polynomial $\sum_{i=
0}^n d_i (1+x)^i$ \cite{Hog74}. (ii) follows from the next fact.

\begin{f}\label{ef}
Given a nonnegative sequence $\{d_i\}_{i=0}^n$, if the subsequence
$\{d_i:d_i\neq0\}$ is increasing, then the polynomial $\sum_{i= 0}^n
d_i (1+x)^i$ is unimodal.
\end{f}
\textbf{The proof of Fact \ref{ef}:} Let
$f_n(x)=\sum_{i=0}^{n}d_i(1+x)^{i}=\sum_{i=0}^{n}c_ix^i$. Since the
subsequence $\{d_i:d_i\neq0\}$ is increasing, we can assume
$d_n\neq0$. We will show this fact by induction on $n$. If $n=1$,
then it is trivial since $f_1(x)=d_0+d_1(x+1)=d_1x+(d_0+d_1)$. So we
proceed to the inductive steps ($n\geq2$).

Let $F(x)=\sum_{i=0}^{n-1}d_{i+1}x^i$. Then
\begin{eqnarray}\label{f}
f_n(x)=d_0+(1+x)F(1+x).
\end{eqnarray} By the
induction hypothesis, $F(1+x)$ is unimodal. So $(1+x)F(1+x)$ is
unimodal by Lemma~\ref{product} (ii). Thus $c_1,c_2,\ldots,c_{n}$ is
unimodal. On the other hand, note that
$$c_0=\sum_{i=0}^nd_i<\sum_{i=1}^nid_i=c_1$$ since the subsequence $\{d_i:d_i\neq0\}$ is increasing.
It follows that $c_0,c_1,c_2,\ldots,c_{n}$ is still unimodal, i.e.,
$f_n(x)$ is unimodal. This completes the proof.
\end{proof}
\begin{rem}
In fact, our Fact \ref{ef} generalizes the following result of Boros
and Moll~\cite{BM99}: If $P(x)$ is a polynomial with positive
nondecreasing coefficients, then $P(x+1)$ is unimodal.
\end{rem}
\begin{rem}
If the subsequence $\{a_i:a_i\neq0\}$ is not increasing, then
$I(G;x)$ may not be unimodal. For
instance:$$I(K_{{\underbrace{1,\ldots,1}_{26}},8};x)=26(x+1)+(x+1)^8-26=1+34x+\textbf{28}x^2+56x^3+70x^4+56x^5+28x^6+8x^7+x^8$$
is not unimodal.
\end{rem}
\section{Rooted Product of Graphs}
Let $V(G)=\{v_i\}_{i=1}^n$ and $H$ be a rooted graph with the root
$u$. The rooted product $G\overline{\circ} H$ of the graphs $G$ and
$H$ with respect to the ``root'' $u$ is defined as follows: take $n$
copies of $H$, and for every vertex $v_i$ of $G$, identify $v_i$
with the root $u$ of the $i$th copy of $H$, see Godsil and
MacKay~\cite{GM78} for instance.

\begin{center}
\setlength{\unitlength}{1cm}
\begin{picture}(20,2.5)(-1,-2.0)

\thicklines\put(5,0){\line(0,-1){1}}

  \put(5.0,0.3){$v$}\put(5.0,0.0){\circle*{0.2}}

\put(5,-1){\circle*{0.2}}
\put(4.7,-2.0){ $P_2$}

\put(9,0){\circle*{0.2}}

 \put(9.0,0.2){$v$}
 \put(9,0){\line(1,-1){0.7}}\put(9,0){\line(-1,-1){0.7}}

\put(8.3,-0.7){\circle*{0.2}}\put(9.7,-0.7){\circle*{0.2}}
\put(8.7,-2.0){ $P_3$}

\end{picture}
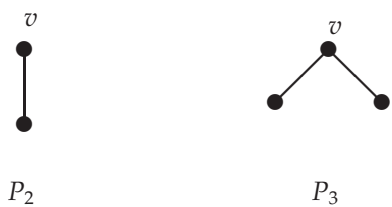
Figure~1.
\end{center}

Let $P_2$ and $P_3$ with the root $v$, respectively, see Figure~1.
For a graph $G$, if $I(G;x)$ has only real zeros, then so do
$I(G\overline{\circ} P_2;x)$ and $I(G\overline{\circ} P_3;x)$, see
Levit and Mandrescu \cite{LM07} and Mandrescu~\cite{Man09},
respectively. More generally, let $H$ be a claw-free graph with the
root $v$. If $I(G;x)$ has only real zeros, then so does
$I(G\overline{\circ} H;x)$, see Zhu \cite[Proposition 3.3]{Zhu13}.
Thus, naturally, it should be considered the graphs with claws. If
$H$ has claws, then we give the following special result.

\begin{center}
\setlength{\unitlength}{1cm}
\begin{picture}(20,2.5)(-7,-1)

\thicklines \put(-3,0){\line(1,0){2}}

\put(-3,0){\circle*{0.2}}\put(-2,0){\circle*{0.2}}
\put(-1,0){\circle*{0.2}}

 \put(-3,0){\line(0,1){1}} \put(-3,0){\line(0,-1){1}}

 \put(-3,-1){\circle*{0.2}}\put(-3,1){\circle*{0.2}}
 \put(-1.1,-0.5){4}\put(-2.1,-0.5){2}\put(-2.9,-0.5){3}\put(-2.9,0.6){1}
\put(-2,-1.5){ $T$}

\thicklines \put(3,0){\line(1,0){3}}

\put(3,0){\circle*{0.2}}\put(4,0){\circle*{0.2}}
\put(5,0){\circle*{0.2}}\put(6,0){\circle*{0.2}}

 \put(3,0){\line(0,1){1}} \put(3,0){\line(0,-1){1}}

 \put(3,-1){\circle*{0.2}}\put(3,1){\circle*{0.2}}
 \put(4.9,-0.5){3}\put(3.9,-0.5){2}\put(6,-0.5){4}\put(3.1,0.6){1}
 \put(4,-1.5){ $T_1$}

\end{picture}\\[5mm]
Figure~2.
\end{center}

\begin{prop}\label{prop+real}
Let the graphs $T$ and $T_1$ be in Figure $2$ with the root $v$. If
$I(G;x)$ has only real zeros, then we have the following.
\begin{itemize}
\item [\rm (i)]
$I(G\overline{\circ} T;x)$ has only real zeros for $v\in\{1,2,3\}$
and $I(G\overline{\circ} T;x)$ is log-concave for $v=4$;
\item [\rm (ii)]
$I(G\overline{\circ} T_1;x)$ is log-concave for $v\in\{1,2,3,4\}$.
\end{itemize}
\end{prop}
\begin{proof} Since the proofs are similar, for brevity we only prove (i) for the root being $1$ or
$4$. Recall the formula for independence polynomials of the rooted
product of graphs, see \cite{Gut92,R09} for instance: If $G$ is a
graph of order $n$ and $H$ is a graph with the root $v$, then
$$I(G \overline{\circ}
H;x)=I^n(H-v;x)I\left(G;\frac{xI(H-N[v];x)}{I(H-v;x)}\right).$$
Since $I(G;x$ has only real zeroes, we can assume that
$$I(G;x)=\prod_{i=1}^{\alpha(G)}\left(1+a_ix\right),$$
where $a_i>0$ for $1\leq i\leq \alpha(G)$. Thus
\begin{eqnarray}\label{eq}
I(G \overline{\circ}
T;x)&=&I^n(T-v;x)I\left(G;\frac{xI(T-N[v];x)}{I(T-v;x)}\right)\nonumber\\
&=&I^{n-\alpha(G)}(T-v;x)\prod_{i=1}^{\alpha(G)}\left[I(T-v;x)+a_ixI(T-N[v];x)\right],
\end{eqnarray}

If the root $v=1$, then $I(T-v;x)=(1+x)(1+3x)$ and
$I(T-N[v];x)=(1+x)(1+2x)$. Thus, by (\ref{eq}), we have
\begin{eqnarray}\label{eqq}
I(G \overline{\circ} T;x)
&=&(1+x)^n(1+3x)^n\prod_{i=1}^{\alpha(G)}\left(1+\frac{a_ix(1+2x)}{1+3x}\right)\nonumber\\
&=&(1+x)^n(1+3x)^{n-\alpha(G)}\prod_{i=1}^{\alpha(G)}\left[1+3x+a_ix(1+2x)\right]\nonumber\\
&=&(1+x)^n(1+3x)^{n-\alpha(G)}\prod_{i=1}^{\alpha(G)}\left[1+(3+2a_i)x+2a_ix^2\right].
\end{eqnarray}
It is also easy to confirm that $1+(3+2a_i)x+2a_ix^2$ has only real
zeros for $a_i> 0$. Hence $I(G \overline{\circ} T;x)$ has only real
zeros by (\ref{eqq}) and Lemma \ref{product} (iii).

If the root $v=4$, then $I(T-v;x)=(1+x)^3+x$ and
$I(T-N[v];x)=(1+x)^2+x$. Then \begin{eqnarray}\label{eqqq} I(G
\overline{\circ} T;x)
&=&[(1+x)^3+x]^n\prod_{i=1}^{\alpha(G)}\left(1+\frac{a_ix[(1+x)^2+x]}{(1+x)^3+x}\right)\nonumber\\
&=&[(1+x)^3+x]^{n-\alpha(G)}\prod_{i=1}^{\alpha(G)}\left[(a_i+1)x^3+3(1+a_i)x^2+(3+a_i)x+1\right].
\end{eqnarray}
So, it is easy to obtain the log-concavity of $(1+x)^3+x$ and we
claim that for any positive $r$,
$$(r+1)x^3+3(1+r)x^2+(3+r)x+1$$
is log-concave. Actually, it suffices to prove the inequalities
$$9(1+r)^2-(r+1)(3+r)=(r+1)(8r+6)>0$$
and $$(3+r)^2-3(1+r)=r^2+3r+6>0.$$ Thus it follows from (\ref{eqqq})
and Lemma \ref{product} (i) that $I(G\overline{\circ} T;x)$ is
log-concave. This completes the proof.
\end{proof}

\begin{rem}
If we take a tree $G$ with independence polynomial having only real
zeros, then we can repeatedly use Propositions \ref{prop+real} to
generate infinite trees with unimodal independence polynomials. In
addition, all of our constructions further support Conjecture
\ref{conj+tree}.
\end{rem}
\begin{rem}
Let $(I(H-v;x),I(H-N[v];x))=f(x)(g(x),h(x))$, where $(g(x),h(x))=1$.
Assume that $f(x),g(x),h(x)$ have only real zeros. From the proof,
we can see that if $g(x)+rxh(x)$ has only real zeros for any
positive $r$, then we can obtain that $I(G\overline{\circ} H;x)$ has
only real zeros by Lemma \ref{product} (iii). Generally speaking,
two useful approaches are to guarantee that the zeros of $g(x)$ and
$h(x)$ interlace or the polynomials $g(x)$ and $h(x)$ are
compatible, see Liu and Wang \cite{LW07} and Chudnovsky and
Seymour~\cite{CS07}, respectively. On the other hand, our results
can be generalized to another operation of graphs called the clique
cover product, see Zhu \cite{Zhu13}.
\end{rem}

\section{An Affirmative Answer to Problem~\ref{pro}}
In this section, we answer the Problem~\ref{pro} by finding a
bipartite graph. Define $H_n$ and $G_n$ be the graphs in Figure~1,
where $H_0=\emptyset$, $H_1=K_2$, $G_0=K_1$ and $G_1=K_{1,2}$.
\begin{center}
\setlength{\unitlength}{1cm}
\begin{picture}(22,2.5)(-0.5,-2.0)

\put(3.5,-2.0){$H_n$}
\thicklines \put(1,0){\line(1,0){2.5}}
\thicklines\put(1,1){\line(1,0){2.5}}
 \thicklines \put(1,0){\line(0,1){1}}\thicklines \put(2,0){\line(0,1){1}}\put(3,0){\line(0,1){1}}
 \thicklines\put(6,1){\line(-1,0){1.3}}\thicklines\put(6,1){\line(0,-1){1}}\thicklines\put(6,0){\line(-1,0){1.3}}\thicklines\put(5,1){\line(0,-1){1}}

 \put(0.9,-0.5){$1$}\put(1.9,-0.5){$2$}\put(2.9,-0.5){$3$}\put(4.5,-0.5){$n-1$}\put(5.9,-0.5){$n$}
\put(1,0){\circle*{0.2}}\put(2,0){\circle*{0.2}}\put(3,0){\circle*{0.2}}\put(3.5,0){\circle*{0.1}}\put(3.7,0){\circle*{0.1}}\put(3.9,0){\circle*{0.1}}
\put(4.1,0){\circle*{0.1}}\put(4.1,1){\circle*{0.1}}
\put(5,0){\circle*{0.2}}\put(6,0){\circle*{0.2}}\put(3.5,1){\circle*{0.1}}\put(3.7,1){\circle*{0.1}}\put(3.9,1){\circle*{0.1}}

\put(1,1){\circle*{0.2}}\put(2,1){\circle*{0.2}}\put(3,1){\circle*{0.2}}
\put(6,1){\circle*{0.2}}\put(5,1){\circle*{0.2}}\put(14,0.7){$u$}
\thicklines \put(8,0){\line(1,0){2.5}}
\thicklines\put(8,1){\line(1,0){2.5}}
 \thicklines \put(8,0){\line(0,1){1}}\thicklines \put(9,0){\line(0,1){1}}\put(10,0){\line(0,1){1}}
 \thicklines\put(14,1){\line(-1,0){2.3}}\thicklines\put(13,1){\line(0,-1){1}}\thicklines\put(13,0){\line(-1,0){1.3}}\thicklines\put(12,1){\line(0,-1){1}}

 \put(7.9,-0.5){$1$}\put(8.9,-0.5){$2$}\put(9.9,-0.5){$3$}\put(11.5,-0.5){$n-1$}\put(12.9,-0.5){$n$}
\put(8,0){\circle*{0.2}}\put(9,0){\circle*{0.2}}\put(10,0){\circle*{0.2}}\put(10.5,0){\circle*{0.1}}\put(10.7,0){\circle*{0.1}}\put(10.9,0){\circle*{0.1}}
\put(12,0){\circle*{0.2}}\put(13,0){\circle*{0.2}}\put(10.5,1){\circle*{0.1}}\put(10.7,1){\circle*{0.1}}\put(10.9,1){\circle*{0.1}}
\put(11.1,0){\circle*{0.1}}\put(11.1,1){\circle*{0.1}}
\put(8,1){\circle*{0.2}}\put(9,1){\circle*{0.2}}\put(10,1){\circle*{0.2}}
\put(13,1){\circle*{0.2}}\put(12,1){\circle*{0.2}}\put(14,1){\circle*{0.2}}
 \put(11.5,-2.0){$G_n$}
\end{picture}
Figure~3
\end{center}
The following result is a special case of Corollary 2.4 in Liu and
Wang~\cite{LW07}.
\begin{lem}\label{lem-LW}
Let $\{Q_n(x)\}_{n\ge 0}$ be a sequence of polynomials with
nonnegative coefficients such that
\begin{itemize}
\item [\rm (i)]
$Q_{n}(x)=a_n(x)Q_{n-1}(x)+c_n(x)Q_{n-2}(x)$ for $n\ge 2$.
\item [\rm (ii)]
$Q_0(x)$ is a constant and $\deg Q_{n-1}\le\deg Q_n\le\deg
Q_{n-1}+1$.
\end{itemize}
If $c_n(x)\le 0$ whenever $x\le 0$, then $\{Q_n(x)\}$ has only real
zeros. Furthermore, the zeros of $Q_n(x)$ are separated by the zeros
of $Q_{n-1}(x)$.
\end{lem}

 The next result gives an answer to Problem~\ref{pro}.
\begin{thm}
Let $G_n$ be the graph in Figure~3. Then $I(G_n;x)$ is symmetric and
has only real zeros.
\end{thm}
\begin{proof}
Let $H_n$ be the graph in Figure~3. Then
\begin{eqnarray}\label{eeq}
I(G_n;x)&=&I(G_n-u;x)+xI(G_{n}-N[u];x)\nonumber\\
&=&I(H_n;x)+xI(G_{n-1};x)\nonumber\\
&=&I(G_{n-1};x)+xI(G_{n-2};x)+xI(G_{n-1};x)\nonumber\\
&=&(x+1)I(G_{n-1};x)+xI(G_{n-2};x)
\end{eqnarray}
for $n\geq2$. Note that $I(G_0;x)=1+x$ and $I(G_1;x)=1+3x+x^2$. In
fact, we can set $I(G_{-1};x)=1$, which is well-defined extension by
(\ref{eeq}). Thus, by Lemma~\ref{lem-LW}, $I(G_n;x)$ has only real
zeros. It is not hard to find that the degree of $I(G_n;x)$ is
$n+1$, i.e., $\alpha(G_n)=n+1$.

In the following, we will show that $I(G_n;x)$ is symmetric by
induction $n$. It is obvious for $n=0,1$. Assume that $I(G_k;x)$ is
symmetric for $k\leq n-1$.

To prove the symmetry of $I(G_n;x)$, it suffices to show
$x^{n+1}I(G_n;1/x)=I(G_n;x).$ By (\ref{eq}) and the induction
hypothesis, it follows that
\begin{eqnarray*}
x^{n+1}I(G_n;1/x)&=&x^{n+1}\left[(1/x+1)I(G_{n-1};1/x)+(1/x)I(G_{n-2};1/x)\right]\\
&=&(x+1)I(G_{n-1};x)+xI(G_{n-2};x)\\
&=&I(G_n;x).
\end{eqnarray*}
Thus $I(G_n;x)$ is symmetric. This completes the proof.
\end{proof}
\begin{rem}
Using the method in \cite{WZ10} to solve the linear recurrence
relation (\ref{eeq}), we can also obtain that
\begin{eqnarray}\label{ff}
I(G_n;x)&=&\frac{{\la_1}^{n+2}-{\la_2}^{n+2}}{\la_1-\la_2}\nonumber\\
&=&(1+x)^{\delta_n}\prod_{s=1}^{\lceil{n/2}\rceil}\left[(1+x)^2+4x\cos^2\frac{s\pi}{n+2}\right]\nonumber\\
&=&(1+x)^{\delta_n}\prod_{s=1}^{\lceil{n/2}\rceil}\left[x^2+2x\cos\frac{2s\pi}{n+2}+1\right],
\end{eqnarray}
where $\delta_n=1$ for even $n$ and $0$ otherwise, $\la_1$ and
$\la_2$ are the roots of quadric equation $\la^2-(x+1)\la-x=0$.
Noting that reality of zeros and symmetry of polynomials is closed
under the product of polynomials, respectively, it clearly follows
from (\ref{ff}) that $I(G_n;x)$ is symmetric and has only real
zeros.
\end{rem}


\end{document}